\documentclass[reqno,oneside,11pt]{amsart}

\usepackage[T1]{fontenc}
\usepackage{times,mathptm}
\usepackage{amssymb,epsfig,verbatim,mathtools}
\theoremstyle{plain}

\newtheorem{thm}{Theorem}[section]
\newtheorem{cor}[thm]{Corollary}
\newtheorem{pro}[thm]{Proposition}
\newtheorem{lem}[thm]{Lemma}
\newtheorem{claim}[thm]{Claim}
\newtheorem{theoalph}{Theorem}
\newtheorem{proalph}[theoalph]{Proposition}
\newtheorem{remalph}[theoalph]{Remark}

\theoremstyle{definition}

\newtheorem{rem}[thm]{Remark}

\addtocounter{section}{0}             
\numberwithin{equation}{section}       

\begin{document}

\setlength{\baselineskip}{0.54cm}        
%
%
\title{The embeddings of the Heisenberg group into the \textsc{Cremona} group}
\author{Julie D\'eserti}\thanks{Supported by the ANR grant Fatou ANR-17-CE40-
0002-01 and the ANR grant Foliage ANR-16-CE40-0008-01.}
\address{Universit\'e C\^ote d'Azur, Laboratoire J.-A. Dieudonn\'e, UMR 7351, Nice, France}
\email{deserti@math.cnrs.fr}



\begin{abstract} 
In this note we describe the embeddings of 
the Heisenberg group into the Cremona group.
\end{abstract}

 \maketitle

\section*{Introduction}

The Heisenberg group is the non-abelian nilpotent group
given by
\[
\mathcal{H}=\langle \mathrm{f},\,\mathrm{g}\,\vert\,[\mathrm{f},\mathrm{g}]=\mathrm{h},\,[\mathrm{f},\mathrm{h}]=[\mathrm{g},\mathrm{h}]=\mathrm{id}\rangle.
\]

It has two generators, $\mathrm{f}$ and $\mathrm{g}$, and $\mathrm{h}$ is the 
generator of the center of $\mathcal{H}$. 

The Cremona group is the group 
$\mathrm{Bir}(\mathbb{P}^2_\mathbb{C})$ of 
birational maps of the projective 
plane~$\mathbb{P}^2_\mathbb{C}$ into itself. Such maps
can be written in the form
\[
(x:y:z)\dashrightarrow\big(P_0(x,y,z):P_1(x,y,z):P_2(x,y,z)\big)
\]
where $P_0$, $P_1$, $P_2\in\mathbb{C}[x,y,z]$ are
homogeneous polynomials of the same degree, and 
this degree is the degree of the map, if the 
polynomials have no common factor (of positive degree).
Recall that if $\phi$ is a birational self map of the complex
projective plane, then one of the following 
holds (\cite{Gizatullin, Cantat, DillerFavre, BlancDeserti}):

\begin{itemize}
\item[$\diamond$] the sequence $(\deg(\phi^n))_{n\in\mathbb{N}}$
is bounded,
and $\phi$ is said to be \textsl{elliptic};

\item[$\diamond$] the sequence $(\deg(\phi^n))_{n\in\mathbb{N}}$
grows linearly with $n$, and $\phi$ is said to be a
\textsl{Jonqui\`eres twist}; 

\item[$\diamond$] the sequence $(\deg(\phi^n))_{n\in\mathbb{N}}$
grows quadratically with $n$, and $\phi$ is said to be a
\textsl{Halphen twist}; 

\item[$\diamond$] $(\deg(\phi^n))_{n\in\mathbb{N}}$ grows
exponentially fast 
with $n$, and $\phi$ is said to be 
\textsl{hyperbolic}.
\end{itemize}

\begin{proalph}\label{proalpha}
Let $\rho$ be an embedding of $\mathcal{H}$
into the Cremona group. Then
$\rho(\mathcal{H})$ does not contain 
hyperbolic birational maps.

More precisely $\rho(\mathrm{f})$ and 
$\rho(\mathrm{g})$ are either elliptic
birational maps, or Jonqui\`eres twists.
\end{proalph}

We describe the embeddings of 
$\mathcal{H}$ into the Cremona group. In 
\cite{Deserti:IMRN} we already looked at such 
embeddings but 
with the following assumption: the images of 
$\mathrm{f}$ and $\mathrm{g}$ are elliptic 
birational self maps. There are other 
embeddings of $\mathcal{H}$ into the  
Cremona group:

\begin{theoalph}\label{thm:main}
Let $\rho$ be an embedding of $\mathcal{H}$ into the  
Cremona group. 
Then up to birational conjugacy
\begin{itemize}
\item[$\diamond$] either $\rho(\mathcal{H})$ is a subgroup of 
$\mathrm{PGL}(3,\mathbb{C})$ and
\begin{align*}
&\rho(\mathrm{f})=(x+\alpha y,y+\beta) 
&&\rho(\mathrm{g})=(x+\gamma y,y+\delta)
\end{align*}
with $\alpha$, $\beta$, $\gamma$, $\delta$ in $\mathbb{C}$
such that $\alpha\delta-\beta\gamma=1$; 

\item[$\diamond$] or $\rho(\mathcal{H})$ is a subgroup of 
the group of polynomial automorphisms of $\mathbb{C}^2$
and $(\rho(\mathrm{f}),\rho(\mathrm{g}))$
is one of the following pairs
\begin{align*}
& \big((ax+Q(y),y+c),\,(\alpha x+P(y),y+\gamma)\big), \\
&\Big(\left(ax+Q(y),by+\frac{\gamma(b-1)}{\beta-1}\right),\big(\alpha x+P(y),\beta y+\gamma\big)\Big) 
\end{align*}
with $\beta\in\mathbb{C}^*\smallsetminus\{1\}$, $a$, $\alpha$, 
$b$ in $\mathbb{C}^*$, $c$, $\gamma$ in $\mathbb{C}$
and $P$, $Q$ in $\mathbb{C}[y]$;

\item[$\diamond$] or $\rho(\mathrm{f})$ is a Jonqui\`eres twist
and $(\rho(\mathrm{f}),\rho(\mathrm{g}))$
is one of the following pairs
\begin{align*}
& \Big(\big(x,\delta x^{\pm 1}y\big),\,\big(\gamma x,ya(x)\big)\Big),
&& \Big(\big(x,\delta x^{\pm 2}y\big),\,\big(\gamma x,ya(x)\big)\Big),\\
&  \big((-x,\delta x^{\pm 1}y),\,(\gamma x,yb(x)) \big),
&&  \big((\lambda x,yc(x)),(\delta x,yd(x))\big) 
\end{align*}
with $\delta$, $\gamma\in\mathbb{C}^*$, $\lambda\in\mathbb{C}^*\smallsetminus\{1,\,-1\}$
 and
$a$, $b$, $c$, $d\in\mathbb{C}(x)^*$ such that
\begin{align*}
& \frac{b(x)}{b(-x)}\in\mathbb{C}^*,
&& \frac{c(\delta x)d(x)}{c(x)d(\lambda x)}\in\mathbb{C}^*.
\end{align*}
\end{itemize}
\end{theoalph}

\medskip

\begin{remalph}
Note that the  two last families are not 
empty. For instance 
\begin{align*}
&  \big((-x,\alpha x^{\pm 1}y),\,(\beta x,\gamma x^2y) \big),
&&  \big((\lambda x,\alpha x^py),(\gamma x,\beta x^qy)\big) &&
\end{align*}
with 
$\alpha$, $\beta$, $\gamma\in\mathbb{C}^*$, 
$\lambda\in\mathbb{C}^*\smallsetminus\{1,\,-1\}$,  
$p$,
$q\in\mathbb{N}$ are such pairs.
\end{remalph}

\section{Some recalls}

\subsection{About birational maps of the complex 
projective plane}\label{subsec:birmap}

Let $\phi$ be a birational self map of the complex
projective plane. Then one of the following 
holds (\cite{Gizatullin, Cantat, DillerFavre, BlancDeserti}):

\begin{itemize}
\item[$\diamond$] $\phi$ is \textsl{elliptic} if and only if 
the sequence $(\deg(\phi^n))_{n\in\mathbb{N}}$ is bounded. In this 
case there exist a birational map 
$\psi\colon S\dashrightarrow\mathbb{P}^2_\mathbb{C}$
and an integer $k\geq 1$ such that 
$\psi^{-1}\circ\phi^k\circ\psi$ belongs to the connected
component of the identity of the group
$\mathrm{Aut}(S)$. Either $\phi$ is of finite order, or $\phi$ is
conjugate to an automorphism of $\mathbb{P}^2_\mathbb{C}$,
which restricts to one of the following automorphisms 
on some open subset isomorphic to $\mathbb{C}^2$:
\begin{enumerate}
\item[$\bullet$] $(x,y)\mapsto(\alpha x,\beta y)$ where the 
kernel of the group morphism
\begin{align*}
& \mathbb{Z}^2\to\mathbb{C}^2 && (i,j)\mapsto \alpha^i\beta^j
\end{align*}
is generated by $(k,0)$ for some $k\in\mathbb{Z}$;

\item[$\bullet$] $(x,y)\mapsto(\alpha x,y+1)$ where 
$\alpha\in\mathbb{C}^*$.
\end{enumerate}

\item[$\diamond$] $\phi$ is \textsl{parabolic} if and only 
if the sequence $(\deg(\phi^n))_{n\in\mathbb{N}}$ grows linearly or 
quadra\-tically with $n$. If $\phi$ is parabolic, 
there exist a birational map 
$\psi\colon S\dashrightarrow\mathbb{P}^2_\mathbb{C}$
and a fibration $\pi\colon S\to B$ onto a curve 
$B$ such that $\psi^{-1}\circ\phi\circ\psi$ 
permutes the fibers of $\pi$. If $(\deg(\phi^n))_{n\in\mathbb{N}}$ 
grows linearly, then the fibration $\pi$ is rational and 
$\phi$ is said to be a \textsl{Jonqui\`eres twist}.
If $(\deg(\phi^n))_{n\in\mathbb{N}}$ 
grows quadratically, then the fibration $\pi$ is elliptic and 
$\phi$ is said to be a \textsl{Halphen twist}.

\item[$\diamond$] $\phi$ is \textsl{hyperbolic} if and only if
$(\deg(\phi^n))_{n\in\mathbb{N}}$ grows exponentially fast with $n$: 
there is a constant $c(\phi)$ such that 
$\deg(\phi^n)=c(\phi)\lambda(\phi)^n+O(1)$.
\end{itemize}

\subsection{About distorted elements}

If $\mathrm{G}$ is a group generated by a finite subset 
$F\subset\mathrm{G}$ the $F$-length $\vert g\vert_F$ 
of an element $g$ of $\mathrm{G}$ is defined as the 
least non-negative integer $\ell$ such that 
$g$ admits an expression of the form 
$g=f_1f_2\ldots f_\ell$ where each $f_i$ belongs to
$F\cup F^{-1}$. We say that $g$ is 
\textsl{distorted} if 
$\displaystyle\lim_{k\to +\infty}\frac{\vert g^k\vert_F}{k}=0$ 
(note that the limit 
$\displaystyle\lim_{k\to +\infty}\frac{\vert g^k\vert_F}{k}$ 
always exists and is a real
number since the sequence $k\mapsto \vert g^k\vert_F$ is subadditive). 
This notion actually does not depend on the chosen 
$F$, but only on the pair $(g,\mathrm{G})$.

If $\mathrm{G}$ is any group, an element $g\in\mathrm{G}$
is \textsl{distorted} if it is distorted in some 
finitely generated subgroup of $\mathrm{G}$.

The element $\mathrm{h}$ of
\[
\mathcal{H}=\langle \mathrm{f},\,\mathrm{g}\vert\,[\mathrm{f},\mathrm{g}]=\mathrm{h},\,[\mathrm{f},\mathrm{h}]=[\mathrm{g},\mathrm{h}]=\mathrm{id}\rangle
\]
satisfies the following property:
\[
\forall\,k\in\mathbb{Z}\qquad\mathrm{h}^{k^2}=[\mathrm{f}^k,\mathrm{g}^k]=\mathrm{f}^k\mathrm{g}^k\mathrm{f}^{-k}\mathrm{g}^{-k}
\]
so $\vert \vert \mathrm{h}^{k^2}\vert\vert\leq 4k$ and 
$\displaystyle\lim_{k\to +\infty}\frac{\vert\vert \mathrm{h}^k\vert\vert}{k^2}=0$.
Hence $\mathrm{h}$ is distorted.

An element $\phi\in\mathrm{Bir}(\mathbb{P}^2_\mathbb{C})$ is 
said to be \textsl{algebraic}
if it is contained in an algebraic subgroup of 
$\mathrm{Bir}(\mathbb{P}^2_\mathbb{C})$.
By \cite[\S 2.6]{BlancFurter:annals} the 
map $\phi\in\mathrm{Bir}(\mathbb{P}^2_\mathbb{C})$
is algebraic if and only if the sequence
$(\deg(\phi^n))_{n\in\mathbb{N}}$ is bounded.
In other words elliptic elements and algebraic 
elements coincide.
By \cite[Proposition 2.3]{BlancDeserti} this is 
also equivalent to say that $\phi$ is of finite 
order or conjugate to an element of 
$\mathrm{Aut}(\mathbb{P}^2_\mathbb{C})$.
A straightforward computation shows that 
every element of $\mathrm{Aut}(\mathbb{P}^2_\mathbb{C})$
is distorted in $\mathrm{Bir}(\mathbb{P}^2_\mathbb{C})$
(\emph{see} \cite[Lemma 4.40]{BlancFurter:lebesgue}).
As a consequence every algebraic element of 
$\mathrm{Bir}(\mathbb{P}^2_\mathbb{C})$ is distorted. 
The converse statement also holds:

\begin{thm}[\cite{BlancFurter:lebesgue, CantatCornulier}]\label{thm:distorted}
Any distorted element of $\mathrm{Bir}(\mathbb{P}^2_\mathbb{C})$
is elliptic.
\end{thm}

\begin{cor}\label{cor:dis}
Let $\rho$ be an embedding of $\mathcal{H}$
into $\mathrm{Bir}(\mathbb{P}^2_\mathbb{C})$. Then 
$\rho(\mathrm{h})$ is elliptic.
\end{cor}

We will use this corollary and the following description
of the centralizer of hyperbolic birational maps to 
prove Proposition \ref{proalpha}:

\begin{pro}[\cite{BlancCantat}]\label{pro:blanccantat}
Let $\phi$ be a birational map of the complex 
projective plane. If $\phi$ is hyperbolic, 
then the infinite cyclic group generated by 
$\phi$ is a finite index
subgroup of the centralizer of $\phi$ in 
$\mathrm{Bir}(\mathbb{P}^2_\mathbb{C})$.
\end{pro}

\begin{proof}[Proof of the first part of Proposition \ref{proalpha}]
Assume that $\rho(\mathcal{H})$ contains 
an hyperbolic e\-lement $\phi$. Since 
$\rho(\mathrm{h})$ is the generator of the center of 
$\rho(\mathcal{H})$, $\rho(\mathrm{h})$ 
commutes with $\phi$. Proposition 
\ref{pro:blanccantat} 
implies
that either $\rho(\mathrm{h})$ is 
hyperbolic, or $\rho(\mathrm{h})$ is 
of finite order. But $\rho(\mathrm{h})$ is 
not hyperbolic Corollary (\ref{cor:dis}) and by definition 
$\rho(\mathrm{h})$ is of infinite order.
As a result $\rho(\mathcal{H})$ does not
contain hyperbolic element.
\end{proof}

\subsection{About centralizers of elliptic birational maps}\label{subsec:centr}

Let us recall the description of the centralizers of the 
elliptic birational self maps of infinite order 
of the complex projective plane obtained in 
\cite{BlancDeserti}.

Consider $\phi$ an elliptic element of 
$\mathrm{Bir}(\mathbb{P}^2_\mathbb{C})$. 
Assume that $\phi$ is
of infinite order. As recalled in 
\S\ref{subsec:birmap} the map $\phi$ is conjugate to an 
automorphism of 
$\mathbb{P}^2_\mathbb{C}$ which restricts to one of
the following automorphisms on some open subset 
isomorphic to $\mathbb{C}^2$:
\begin{itemize}
\item[(1)] $(\alpha x,\beta y)$ 
where $\alpha$, $\beta$ belong to $\mathbb{C}^*$;

\item[(2)]  $(\alpha x,y+1)$
where $\alpha\in\mathbb{C}^*$.
\end{itemize}

In case (1) the centralizer of $\phi$ in 
$\mathrm{Bir}(\mathbb{P}^2_\mathbb{C})$ is
\[
\big\{(\eta(x),ya(x^k))\,\vert\, a\in\mathbb{C}(x),\,\eta\in\mathrm{PGL}(2,\mathbb{C}),\,\eta(\alpha x)=\alpha\eta(x)\big\};
\]
in particular the elements of the centralizer of
$\phi$ are elliptic birational maps or 
Jonqui\`eres twists. 
In case (2) the centralizer of $\phi$ in 
$\mathrm{Bir}(\mathbb{P}^2_\mathbb{C})$ is
\[
\big\{(\eta(x),y+a(x))\,\vert\,\eta\in\mathrm{PGL}(2,\mathbb{C}),\,\eta(\alpha x)=\alpha\eta(x),\,a\in\mathbb{C}(x),\,a(\alpha x)=a(x)\big\};
\]
in particular the elements of the centralizer of
$\phi$ are elliptic birational maps.

Corollary \ref{cor:dis} and the previous description 
imply:

\begin{lem}\label{lem:heisenberg}
Let $\rho$ be an embedding of  
$\mathcal{H}$ into $\mathrm{Bir}(\mathbb{P}^2_\mathbb{C})$.

Then $\rho(\mathrm{h})$ is elliptic and up to birational conjugacy
\smallskip
\begin{itemize}
\item[$\diamond$] either $\rho(\mathrm{h})=(\alpha x,\beta y)$, 
where the kernel of the group morphism
\begin{align*}
& \mathbb{Z}^2\to\mathbb{C}^2 && (i,j)\mapsto \alpha^i\beta^j
\end{align*}
is generated by $(k,0)$ for some $k\in\mathbb{Z}$ and both 
$\rho(\mathrm{f})$, $\rho(\mathrm{g})$ belong to
\[
\big\{(\eta(x),ya(x^k))\,\vert\, a\in\mathbb{C}(x),\,\eta\in\mathrm{PGL}(2,\mathbb{C}),\,\eta(\alpha x)=\alpha\eta(x)\big\};
\]

\item[$\diamond$] or $\rho(\mathrm{h})=(\alpha x,y+1)$, 
and both $\rho(\mathrm{f})$, $\rho(\mathrm{g})$ belong to 
\[
\big\{(\eta(x),y+a(x))\,\vert\,\eta\in\mathrm{PGL}(2,\mathbb{C}),\,\eta(\alpha x)=\alpha\eta(x),\,a\in\mathbb{C}(x),\,a(\alpha x)=a(x)\big\}.
\]
\end{itemize} 

In particular $\rho(\mathrm{f})$ and $\rho(\mathrm{g})$ are elliptic 
birational maps or Jonqui\`eres twists.
\end{lem}

It ends the proof of Proposition \ref{proalpha}.

\section{Proof of Theorem \ref{thm:main}}

\subsection{Assume that all the generators of $\rho(\mathcal{H})$ are elliptic}

The group $\mathrm{Aut}(\mathbb{C}^2)$ of polynomial 
automorphisms of $\mathbb{C}^2$ is a subgroup of
$\mathrm{Bir}(\mathbb{P}^2_\mathbb{C})$. It is 
generated by the group
\[
\mathrm{A}=\big\{(a_0x+a_1y+a_2,b_0x+b_1y+b_2)\,\vert\,a_i,\,b_i\in\mathbb{C},\,a_0b_1-a_1b_0\not=0\big\} 
\]
and 
\[
\mathrm{E}=\big\{(\alpha x+P(y),\beta y+\gamma)\,\vert\,\alpha,\,\beta\in\mathbb{C}^*,\,\gamma\in\mathbb{C},\, P\in\mathbb{C}[y]\big\}.
\]

Let us recall the following result obtained when we
study the embeddings of $\mathrm{SL}(n,\mathbb{Z})$
into the Cremona group:

\begin{lem}[\cite{Deserti:IMRN}]\label{lem:IMRN}
Let $\rho$ be an embedding of $\mathcal{H}$
into $\mathrm{Bir}(\mathbb{P}^2_\mathbb{C})$.

If $\rho(\mathrm{f})$, $\rho(\mathrm{g})$ and $\rho(\mathrm{h})$
are elliptic, then up to birational conjugacy
\begin{itemize}
\item[$\diamond$] either $\rho(\mathcal{H})$ is a subgroup of 
$\mathrm{PGL}(3,\mathbb{C})$, and more precisely 
\begin{align*}
&\rho(\mathrm{f})=(x+\alpha y,y+\beta) 
&&\rho(\mathrm{g})=(x+\gamma y,y+\delta) 
\end{align*}
with $\alpha$, $\beta$, $\gamma$, 
$\delta\in\mathbb{C}$ such that $\alpha\delta-\beta\gamma=1$; 

\item[$\diamond$] or $\rho(\mathcal{H})$ is a subgroup of 
$\mathrm{E}$
and $\rho(h^2)=(x+P(y),y)$ for some $P\in\mathbb{C}[y]$.
\end{itemize}
\end{lem}

This statement implies the following one:

\begin{pro}
Let $\rho$ be an embedding from $\mathcal{H}$
into $\mathrm{Bir}(\mathbb{P}^2_\mathbb{C})$.
Assume that $\rho(\mathrm{f})$, $\rho(\mathrm{g})$ and $\rho(\mathrm{h})$
are elliptic. 

Then up to birational conjugacy
\begin{itemize}
\item[$\diamond$] either $\rho(\mathcal{H})$ is a subgroup of 
$\mathrm{PGL}(3,\mathbb{C})$, more precisely 
\begin{align*}
&\rho(\mathrm{f})=(x+\alpha y,y+\beta) 
&&\rho(\mathrm{g})=(x+\gamma y,y+\delta)
\end{align*}
with $\alpha$, $\beta$, $\gamma$, $\delta\in\mathbb{C}$ 
such that $\alpha\delta-\beta\gamma=1$;

\item[$\diamond$] or $\rho(\mathcal{H})$ is a subgroup of 
$\mathrm{E}$ and $(\rho(\mathrm{f}),\rho(\mathrm{g}))$ is one of the 
following pairs
\begin{align*}
& \big((ax+Q(y),y+c),\,(\alpha x+P(y),y+\gamma)\big) \\
&\Big(\left(ax+Q(y),by+\frac{\gamma(b-1)}{\beta-1}\right),\big(\alpha x+P(y),\beta y+\gamma\big)\Big) 
\end{align*}
with $a$, $\alpha$, 
$b$ in $\mathbb{C}^*$, $c$, $\gamma$ in $\mathbb{C}$, 
$\beta\in\mathbb{C}^*\smallsetminus\{1\}$
and $P$, $Q$ in $\mathbb{C}[y]$.
\end{itemize}
\end{pro}

\begin{proof}
The first assertion follows from Lemma \ref{lem:IMRN}. Let
us focus on the second one.

If $\rho(\mathrm{h})$ belongs to $\mathrm{E}$ and $\rho(\mathrm{h}^2)=(x+P(y),y)$, 
then $\rho(\mathrm{h})=(\varepsilon x+Q(y),\eta(y))$ with $\varepsilon^2=1$,
$Q\in\mathbb{C}[y]$ and $\eta(y)\in\big\{-y+\gamma,\,y\big\}$.
But $\rho(\mathrm{f})$ and $\rho(\mathrm{g})$ belong to 
$\mathrm{E}$ so $[\rho(\mathrm{f}),\rho(\mathrm{g})]=\rho(\mathrm{h})$ implies that 
$\epsilon=1$ and $\eta(y)=y$, {\it i.e.} $\rho(\mathrm{h})=(x+Q(y),y)$.
Set 
\begin{align*}
&\rho(\mathrm{f})=(ax+R(y),by+c),
&& \rho(\mathrm{g})=(\alpha x+P(y),\beta y+\gamma). 
\end{align*}
The second component 
of $\rho(\mathrm{f})\rho(\mathrm{g})$ has to be equal to the second component of
$\rho(\mathrm{h})\rho(\mathrm{g})\rho(\mathrm{f})$, that is 
\[
\beta by+b\gamma+c=\beta by+\beta c+\gamma;
\]
in other words either $\beta=b=1$, or $c=\frac{\gamma(b-1)}{\beta-1}$.
\end{proof}

\subsection{Assume that $\rho(\mathrm{f})$ is a Jonqui\`eres twist
with trivial action on the basis of the fibration}

Since $\rho(\mathrm{h})$ is elliptic, then up to 
birational conjugacy either
$\rho(\mathrm{h})=(\alpha x,\beta y)$, or 
$\rho(\mathrm{h})=(\alpha x,y+1)$ (\emph{see} 
\S\ref{subsec:birmap}). But $\rho(\mathrm{f})$
belongs to the centralizer of 
$\rho(\mathrm{h})$ and is a Jonqui\`eres twist; 
therefore according to \S \ref{subsec:centr} 
one has: $\rho(\mathrm{h})=(\alpha x,\beta y)$, 
$\rho(\mathrm{f})$ can be written as 
$(x,ya(x))$ and $\rho(\mathrm{g})$ as 
$(\mu(x),yb(x))$ with $\mu\in\mathrm{PGL}(2,\mathbb{C})$
and $a$, $b\in\mathbb{C}(x)^*$.

Let us remark that 
if $\mu=\mathrm{id}$, then $[\rho(\mathrm{f}),\rho(\mathrm{g})]=\rho(\mathrm{h})$ implies
$\alpha=\beta=1$ so $\mu\not=\mathrm{id}$.

The relation $[\rho(\mathrm{f}),\rho(\mathrm{g})]=\rho(\mathrm{h})$ implies that $\alpha=1$, and $a(\mu(x))=\beta a(x)$.
Let us first look at polynomials $P$ such that 
$P(\mu(x))=\beta P(x)$:

\begin{claim}
If $P$ is a non-zero polynomial such that $P(\mu(x))=\lambda^2P(x)$, $\lambda^2\not=1$,
then one of the following holds:
\begin{itemize}
\item[$\diamond$] $P(x)=\delta\left(\frac{\gamma}{\lambda^2-1}+x\right)$, $\mu(x)=\gamma+\lambda^2 x$ with $a$, 
$\delta\in\mathbb{C}$;


\item[$\diamond$] $P(x)=\delta\left(\frac{\gamma}{\lambda+1}-x\right)^2$, 
$\mu(x)=\gamma-\lambda x$ with $\gamma\in\mathbb{C}$, and
$\delta\in\mathbb{C}^*$;

\item[$\diamond$] $P(x)=\delta\left(\frac{\gamma}{\lambda-1}+x\right)^2$, 
$\mu(x)=\gamma+\lambda x$ with $\gamma\in\mathbb{C}$, and
$\delta\in\mathbb{C}^*$.
\end{itemize}
\end{claim}

\begin{proof}
Let us consider the set $Z_P=\big\{z\,\vert\, P(z)=0\big\}$ 
of roots of $P$. It is a finite set invariant by $\mu$. As 
a result $\mu^n_{\vert Z_P}=\mathrm{id}$ for some integer
$n$. 

If $\# Z_P\geq 3$, then $\mu^n_{\vert Z_P}=\mathrm{id}$
implies $\mu^n=\mathrm{id}$. 
Recall that 
\begin{align*}
& \rho(\mathrm{f})=(x,ya(x)), && \rho(\mathrm{g})=(\mu(x),yb(x)),  && \rho(\mathrm{h})=(\alpha x,\beta y)
\end{align*}
so 
\begin{align*}
& \rho(\mathrm{f})^n=(x,yA(x)), && \rho(\mathrm{g})^n=(\mu^n(x),yB(x))=(x,yB(x)),  && \rho(\mathrm{h})^{n^2}=(\alpha^{n^2} x,\beta^{n^2} y).
\end{align*}
Then $[\rho(\mathrm{f})^n,\rho(\mathrm{g})^n]=\rho(\mathrm{h})^{n^2}$ implies $\alpha^{n^2}=\beta^{n^2}=1$, 
that is $\rho(\mathrm{h})$ is of finite order: 
contradiction.

Hence $\#Z_P\leq 2$ so $\deg P\leq 2$. 
A straightforward computation implies the statement.
\end{proof}

Let us come back to $a(\mu(x))=\beta a(x)$. As $\beta$ is of 
infinite order and $a$ belongs to~$\mathbb{C}(x)^*$ we can rewrite
this equality as follows: 
$\frac{P(\mu(x))}{Q(\mu(x))}=\frac{\lambda_1^2P(x)}{\lambda_2^2Q(x)}$
where 
\begin{itemize}
\item[$\diamond$] $\lambda_1$, $\lambda_2$ are two elements of
$\mathbb{C}\smallsetminus\{\pm 1\}$ such that $\beta=\frac{\lambda_1^2}{\lambda_2^2}$;
\item[$\diamond$] $P$, $Q$ are two polynomials without common 
factor.
\end{itemize}
As a result up to birational conjugacy 
$(\rho(\mathrm{f}),\rho(\mathrm{g}))$ is one of the 
following pairs
\begin{align*}
& \Big(\big(x,\delta xy\big),\,\big(\lambda x,yb(x)\big)\Big)
&& \Big(\big(x,\delta x^2y\big),\,\big(\lambda x,yb(x)\big)\Big)\\
&\Big(\big(x,\frac{y}{\delta x}\big),\,\big(\lambda x,yb(x)\big)\Big)
&& \Big(\big(x,\frac{y}{\delta x^2}\big),\,\big(\lambda x,yb(x)\big)\Big)
\end{align*}
with $\delta\in\mathbb{C}^*$, $\lambda\in\mathbb{C}^*$ of infinite order 
and $b\in\mathbb{C}(x)$.

We can thus state

\begin{pro}\label{pro:fiberwise}
Let $\rho$ be an embedding of $\mathcal{H}$
into $\mathrm{Bir}(\mathbb{P}^2_\mathbb{C})$.

If $\rho(\mathrm{f})$ is a Jonqui\`eres twist with trivial action
on the basis of the fibration, then up to birational conjugacy
$(\rho(\mathrm{f}),\rho(\mathrm{g}))$ is one of the 
following pairs
\begin{align*}
& \Big(\big(x,\delta xy\big),\,\big(\lambda x,yb(x)\big)\Big)
&& \Big(\big(x,\delta x^2y\big),\,\big(\lambda x,yb(x)\big)\Big)\\
&\Big(\left(x,\frac{y}{\delta x}\right),\,\big(\lambda x,yb(x)\big)\Big)
&& \Big(\left(x,\frac{y}{\delta x^2}\right),\,\big(\lambda x,yb(x)\big)\Big)
\end{align*}
with $\delta\in\mathbb{C}^*$, $\lambda\in\mathbb{C}^*$ of infinite order 
and $b\in\mathbb{C}(x)$.
\end{pro}

\subsection{Assume that $\rho(\mathrm{f})$ is a Jonqui\`eres twist
with non-trivial action on the basis of the fibration}

Since $\rho(\mathrm{h})$ is elliptic and 
of infinite order, then up to 
birational conjugacy either
$\rho(\mathrm{h})=(\alpha x,\beta y)$, or 
$\rho(\mathrm{h})=(\alpha x,y+1)$ (\emph{see} 
\S\ref{subsec:birmap}). But $\rho(\mathrm{f})$
belongs to the centralizer of 
$\rho(\mathrm{h})$ and is a Jonqui\`eres twist; 
therefore according to \S \ref{subsec:centr} 
one has: $\rho(\mathrm{h})=(\alpha x,\beta y)$,
 $\rho(\mathrm{f})$ can be written
as $(\eta(x),ya(x))$ and $\rho(\mathrm{g})$ as $(\mu(x),yb(x))$
with $\eta$, $\mu$ in $\mathrm{PGL}(2,\mathbb{C})$ and 
$a$, $b$ in~$\mathbb{C}(x)$. 

Up to conjugacy by an element of 
$\Big\{\left(\frac{ax+b}{cx+d},y\right)\,\vert\,\left(\begin{array}{cc}a & b\\ c & d\end{array}\right)\in\mathrm{PGL}(2,\mathbb{C})\Big\}$
one can assume that either $\eta(x)=x+1$, or $\eta(x)=\lambda x$
(remark that this conjugacy doesn't preserve the first component 
of $\rho(\mathrm{h})$). 

Note that a direct computation implies 
\begin{equation}\label{eq:cocycle}
\big\{\nu\in\mathrm{PGL}(2,\mathbb{C})\,\vert\,\nu(\alpha x)=\alpha \nu(x)\big\}=\left\{
\begin{array}{lll}
\mathrm{PGL}(2,\mathbb{C})\text{ if $\alpha=1$}\\
\big\{\beta x^{\pm 1}\,\vert\,\beta\in\mathbb{C}^*\big\}\text{ if $\alpha=-1$}\\
\big\{\beta x\,\vert\,\beta\in\mathbb{C}^*\big\}\text{ if $\alpha^2\not=1$}\\
\end{array}
\right.
\end{equation}
so when $\eta$ is an homothety we will have to 
distinguish the cases $\lambda=-1$ and 
$\lambda\not=-1$.

\subsubsection{Assume that $\eta(x)=x+1$}

Since $\rho(\mathrm{f})$ and $\rho(\mathrm{h})$ commute,
$\rho(\mathrm{h})$ can be written as $(x+\gamma,\beta y)$.

If $\gamma\not=0$, then
$[\rho(\mathrm{f}),\rho(\mathrm{h})]=\mathrm{id}$ 
leads to $a(x+\gamma)=a(x)$, that is $a(x)=a\in\mathbb{C}$:
contradiction with the fact that $\rho(\mathrm{f})$
is a Jonqui\`eres twist.

If $\gamma=0$, then $\rho(\mathrm{h})=(x,\beta y)$
and $[\rho(\mathrm{f}),\rho(\mathrm{g})]=\rho(\mathrm{h})$
leads to $\rho(\mathrm{g})=(x+\mu,yb(x))$ and 
$b(x)a(x+\mu)=\beta a(x)b(x+1)$. Let us write 
$a$ as $\frac{P}{Q}$ and $b$ as $\frac{R}{S}$
with $P$, $Q$, $R$, $S\in\mathbb{C}[y]$ then
$b(x)a(x+\mu)=\beta a(x)b(x+1)$ can be 
rewritten
\begin{equation}\label{eq}
P(x+\mu)Q(x)R(x)S(x+1)=\beta P(x)Q(x+\mu)R(x+1)S(x)
\end{equation}
Denote by $p_i$ (resp. $q_\ell$, resp. $r_j$,
resp. $s_k$) the coefficient of the highest term
of $P$ (resp. $Q$, resp. $R$, resp. $S$).
The coefficient of the highest term of the left-hand 
side of $(\ref{eq})$ has to be equal to  
 the coefficient of the highest term of the right-hand 
side of $(\ref{eq})$, that is 
$p_iq_\ell r_js_k=\beta p_iq_\ell r_js_k$.
So $\beta=1$, {\it i.e.} $\rho(\mathrm{h})=(x,y)$:
contradiction.

\subsubsection{Suppose that $\eta(x)=-x$, 
{\it i.e.} $\rho(\mathrm{f})=(-x,ya(x))$}

\begin{rem}\label{rem:rhof}
The map
$\rho(\mathrm{f})^2=(x,ya(x)a(-x))$ is a Jonqui\`eres twist that 
preserves fiberwise the rational fibration $x=$ cst; 
consequently Proposition \ref{pro:fiberwise} says that 
$\rho(\mathrm{f})^2$ is one of the following:
\begin{align*}
& \big(x,\delta xy\big),
&& \big(x,\delta x^2y\big),
&&\left(x,\delta\frac{y}{x}\right),
&& \left(x,\delta\frac{y}{x^2}\right)
\end{align*}
with $\delta\in\mathbb{C}^*$.
Let us try to determine $\rho(\mathrm{f})$. 
If 
$\rho(\mathrm{f})^2=(x,\delta xy)$, 
then we have to consider the equation $a(x)a(-x)=\delta x$.
The right-hand side of this equation is invariant by $x\mapsto -x$ whereas
the left-hand side not, so there is no solution. The same holds if 
$\rho(\mathrm{f})^2=\left(x,\delta\frac{y}{x}\right)$.
Consequently 
$\rho(\mathrm{f})^2$ is one of the following:
\begin{align*}
& \big(x,\delta x^2y\big),
&& \left(x,\delta\frac{y}{x^2}\right)
\end{align*}
with $\delta\in\mathbb{C}^*$ and 
$\rho(\mathrm{f})$ is thus one of the following:
\begin{align*}
& \big(-x,\zeta xy\big),
&& \left(-x,\zeta\frac{y}{x}\right)
\end{align*}
with $\zeta\in\mathbb{C}^*$.
\end{rem}

Since $\mathrm{f}$ and $\mathrm{h}$ commute, then
(\ref{eq:cocycle}) implies that
either $\rho(\mathrm{h})=\left(\frac{\alpha}{x},\beta y\right)$,
or $\rho(\mathrm{h})=(\alpha x,\beta y)$.
Let us consider these two cases.

\begin{itemize}
\item[$\diamond$] Assume first that
$\rho(\mathrm{h})=\left(\frac{\alpha}{x},\beta y\right)$.
Note that $\left(\frac{\alpha}{x},\beta y\right)$ 
does not commute neither to $\big(-x,\zeta xy\big)$, 
nor to $\left(-x,\zeta\frac{y}{x}\right)$:
contradiction with
$[\rho(\mathrm{f}),\rho(\mathrm{h})]=\mathrm{id}$.

\item[$\diamond$] Suppose now that
$\rho(\mathrm{h})=(\alpha x,\beta y)$.

\begin{itemize}
\item[$\bullet$] If $\alpha^2\not=1$, then $[\rho(\mathrm{g}),\rho(\mathrm{h})]=\mathrm{id}$
and $(\ref{eq:cocycle})$ imply that
$\rho(\mathrm{g})=(\gamma x,yb(x))$. Then 
$[\rho(\mathrm{f}),\rho(\mathrm{g})]=\rho(\mathrm{h})$
leads to $\alpha=1$: contradiction with $\alpha^2\not=1$.

\item[$\bullet$] If $\alpha=-1$, that is 
$\rho(\mathrm{h})=(-x,\beta y)$, then according to 
$[\rho(\mathrm{g}),\rho(\mathrm{h})]=\mathrm{id}$
and $(\ref{eq:cocycle})$
we get that either 
$\rho(\mathrm{g})=(\gamma x,yb(x))$, or 
$\rho(\mathrm{g})=\left(\frac{\gamma}{x},yb(x)\right)$.
In both cases the relation
$[\rho(\mathrm{f}),\rho(\mathrm{g})]=\rho(\mathrm{h})$
leads to a contradiction.

\item[$\bullet$] If $\alpha=1$, {\it i.e.}
$\rho(\mathrm{h})=(x,\beta y)$, then 
$[\rho(\mathrm{f}),\rho(\mathrm{g})]=\rho(\mathrm{h})$ 
implies that either $\rho(\mathrm{g})=(\gamma x,yb(x))$
or $\rho(\mathrm{g})=\left(\frac{\gamma}{x},yb(x)\right)$.

First let us assume that $\rho(\mathrm{g})=(\gamma x,yb(x))$.
If $\rho(\mathrm{f})=(-x,\zeta xy)$, then 
$[\rho(\mathrm{f}),\rho(\mathrm{g})]=\rho(\mathrm{h})$
leads to $\gamma b(x)=\beta b(-x)$,
that is $\frac{b(x)}{b(-x)}$ belongs to $\mathbb{C}^*$.
If $\rho(\mathrm{f})=\left(-x,\zeta \frac{y}{x}\right)$,
then 
$[\rho(\mathrm{f}),\rho(\mathrm{g})]=\rho(\mathrm{h})$
implies $b(x)=\beta\gamma b(-x)$,
that is $\frac{b(x)}{b(-x)}$ belongs to $\mathbb{C}^*$.

Finally suppose that 
$\rho(\mathrm{g})=\left(\frac{\gamma}{x},yb(x)\right)$.
If $\rho(\mathrm{f})=(-x,\zeta xy)$, then 
$[\rho(\mathrm{f}),\rho(\mathrm{g})]=\rho(\mathrm{h})$
leads to $\gamma b(x)=\beta x^2b(-x)$. Write 
$b$ as $\frac{P}{Q}$ with $P$, $Q$ in 
$\mathbb{C}[x]$; then $\gamma b(x)=\beta x^2b(-x)$
is equivalent to 
\[
\gamma P(x)Q(-x)=\beta x^2P(-x)Q(x)
\]
and the degree of the left-hand side is 
$\deg P+\deg Q$ whereas the degree of the 
right-hand side is $\deg P+\deg Q+2$: 
contradiction. If 
$\rho(\mathrm{f})=\left(-x,\zeta \frac{y}{x}\right)$, 
then a straightforward computation implies
similarly a contradiction.
\end{itemize}

\end{itemize}

\begin{pro}\label{pro:order2}
Let $\rho$ be an embedding of $\mathcal{H}$
into $\mathrm{Bir}(\mathbb{P}^2_\mathbb{C})$.

If $\rho(\mathrm{f})$ is a Jonqui\`eres twist with a order $2$
action on the basis of the fibration, then up to birational conjugacy
$(\rho(\mathrm{f}),\rho(\mathrm{g}))$ is one of the 
following pairs
\begin{align*}
&  \big((-x,\alpha xy),\,(\beta x,ya(x)) \big),
&&  \left(\left(-x,\alpha \frac{y}{x}\right),\,(\beta x,ya(x))\right)
\end{align*}
with $\alpha$, $\beta\in\mathbb{C}^*$ and
$a\in\mathbb{C}(x)^*$ such that
$\frac{a(x)}{a(-x)}\in\mathbb{C}^*$.
\end{pro}

\subsubsection{Assume that $\eta(x)=\lambda x$, $\lambda^2\not=1$}

Recall that
\begin{align*}
& \rho(\mathrm{f})=(\lambda x,ya(x)), &&\rho(\mathrm{g})=(\mu(x),yb(x)), &&\rho(\mathrm{h})=(\upsilon(x),\beta y)
\end{align*}
with $\lambda$ in $\mathbb{C}^*\smallsetminus\{1,\,-1\}$, 
$\beta$ in $\mathbb{C}^*$, 
$\mu$, $\upsilon$ in $\mathrm{PGL}(2,\mathbb{C})$,
and $a$, $b$ in $\mathbb{C}(x)^*$.

First note that since $\rho(\mathrm{f})$ 
and $\rho(\mathrm{h})$ commute, 
$\upsilon(\lambda x)=\lambda\upsilon(x)$. 
According to $(\ref{eq:cocycle})$, the 
homography $\upsilon$ is an homothety
(recall that $\lambda^2\not=1$): 
$\upsilon(x)=\gamma x$ with 
$\gamma\in\mathbb{C}^*$.

The relations 
$[\rho(\mathrm{f}),\rho(\mathrm{g})]=\rho(\mathrm{h})$,
$[\rho(\mathrm{f}),\rho(\mathrm{h})]=[\rho(\mathrm{g}),\rho(\mathrm{h})]=\mathrm{id}$ 
imply the follo\-wing ones
\begin{equation}\label{eq:rel1}
a(x)=a(\gamma x)
\end{equation}
\begin{equation}\label{eq:rel2}
b(x)=b(\gamma x)
\end{equation}
\begin{equation}\label{eq:rel3}
\mu(\gamma x)=\gamma\mu(x)
\end{equation}
\begin{equation}\label{eq:rel4}
\lambda\mu(x)=\gamma\mu(\lambda x)
\end{equation}
\begin{equation}\label{eq:rel5}
b(x)a(\mu(x))=\beta a(x)b(\lambda x)
\end{equation}

We will distinguish the cases $\gamma=1$, 
$\gamma=-1$, $\gamma^2\not=1$.

\begin{itemize}
\item[$\diamond$] Assume that $\gamma^2\not=1$. Then 
$(\ref{eq:cocycle})$ and $(\ref{eq:rel3})$
lead to $\mu(x)=\mu x$ with $\mu\in\mathbb{C}^*$.
Equation $(\ref{eq:rel4})$ can be rewritten 
$\lambda\mu x=\gamma\mu\lambda x$, that is 
$\gamma=1$: contradiction with the assumption 
$\gamma^2\not=1$.

\item[$\diamond$] Suppose that $\gamma=1$. Then 
$\lambda^2\not=1$, $(\ref{eq:cocycle})$ and 
$(\ref{eq:rel4})$ lead to $\mu(x)=\mu x$
with $\mu\in\mathbb{C}^*$. In other 
words 
\begin{align*}
& \rho(\mathrm{f})=(\lambda x,ya(x)), &&\rho(\mathrm{g})=(\mu x,yb(x))
\end{align*}
with $\lambda$ in $\mathbb{C}^*\smallsetminus\{1,\,-1\}$, 
$\mu$ in $\mathbb{C}^*$, 
$a$, $b$ in $\mathbb{C}(x)$ such that 
$\frac{a(\mu x)b(x)}{a(x)b(\lambda x)}$
belongs to $\mathbb{C}^*$.

\item[$\diamond$] Assume that $\gamma=-1$. Then 
$(\ref{eq:cocycle})$ and $(\ref{eq:rel3})$
imply that $\mu(x)=\mu x^{\pm 1}$ 
with $\mu\in~\mathbb{C}^*$. If $\mu(x)=\mu x$, 
then $(\ref{eq:rel4})$ can be rewritten
$\lambda\mu x=-\lambda\mu x$: contradiction. 
If $\mu(x)=\frac{\mu}{x}$, then 
$(\ref{eq:rel4})$ can be rewritten 
$\frac{\lambda\mu}{x}=-\frac{\mu}{\lambda x}$;
hence $\lambda^2=-1$. 
\begin{itemize}
\item[$\bullet$] If $\lambda=\mathbf{i}$, then 
$\rho(\mathrm{f})=(\mathbf{i}x,ya(x))$ and 
$\rho(\mathrm{f})^4=\big(x,ya(x)a(\mathbf{i}x)a(-x)a(-\mathbf{i}x)\big)$
preserves fiberwise the fibration $x=$ cst. 
According to Proposition \ref{pro:fiberwise}
$\rho(\mathrm{f})^4$ can be written as 
$(x,\delta xy)$, or $(x,\delta x^2y)$, or
$\left(x,\frac{y}{\delta x}\right)$, 
or $\left(x,\frac{y}{\delta x^2}\right)$.
If $\rho(\mathrm{f})^4=(x,\delta xy)$, 
then $\delta x=a(x)a(\mathbf{i}x)a(-x)a(-\mathbf{i}x)$;
but the right-hand side of this equality is 
invariant by $x\mapsto\mathbf{i}x$
whereas the left-hand side is not. As 
a consequence $\rho(\mathrm{f})^4$ can not
be written $(x,\delta xy)$. Similarly one sees 
that $\rho(\mathrm{f})^4$ can not be written
$(x,\delta x^2y)$, 
$\left(x,\frac{y}{\delta x}\right)$, 
and $\left(x,\frac{y}{\delta x^2}\right)$.
Thus $\lambda\not=\mathbf{i}$.

\item[$\bullet$] Similarly one gets that the 
case $\lambda=-\mathbf{i}$ does not happen.
\end{itemize}

\begin{pro}
Let $\rho$ be an embedding of $\mathcal{H}$
into $\mathrm{Bir}(\mathbb{P}^2_\mathbb{C})$.

If $\rho(\mathrm{f})$ is a Jonqui\`eres twist with an 
action on the basis of the fibration that is neither 
trivial, nor of order $2$, then up to birational conjugacy
$(\rho(\mathrm{f}),\rho(\mathrm{g}))$ is one of the 
following pairs
\[
((\lambda x,ya(x)),(\mu x,yb(x)))
\]
with $\lambda\in\mathbb{C}^*\smallsetminus\{1,\,-1\}$, 
$\mu\in\mathbb{C}^*$, $a$, $b\in\mathbb{C}(x)^*$ such that 
$\frac{a(\mu x)b(x)}{a(x)b(\lambda x)}\in\mathbb{C}^*$.
\end{pro}
\end{itemize}

\vspace*{3cm}

\bibliographystyle{alpha}
\bibliography{biblio}

\nocite{}

\end{document}